\documentclass[12pt]{article}
\usepackage{amsmath}
\usepackage{amssymb}
\usepackage{amsthm}
\usepackage{graphicx}
\usepackage{authblk}
\newcommand{\ov}{\overline}

\newcommand{\prob}{{\mathsf{Pr}}}

\newtheorem{thm}{Theorem}
\newtheorem*{thm*}{Theorem}
\newtheorem{lem}{Lemma}
\newtheorem{fact}{Fact}
\newtheorem{conj}{Conjecture}
\newtheorem{ques}{Question}
\newtheorem{defi}{Definition}
\newtheorem{eg}{Example}


\title{Graph metric with no proper inclusion between lines}
\author[1]{Xiaomin Chen}
\author[2]{Guangda Huzhang}
\author[2]{\\Peihan Miao}
\author[2]{Kuan Yang}
\date{February 22, 2014}
\affil[1]{Shanghai Jianshi LTD}
\affil[2]{Shanghai Jiao Tong University}
\begin{document}

\maketitle

\begin{abstract}
In trying to generalize the classic Sylvester-Gallai theorem and De Bruijn-Erd\H{o}s theorem in
plane geometry, lines and closure lines were previously defined for metric spaces and hypergraphs.
Both definitions do not obey the geometric intuition in the sense that two lines (closure lines)
may intersect at more than one point, and one line (closure line) might be the proper subset of another.
In this work, we study the systems where one or both of the configurations are forbidden.
We note that when any two lines intersect in at most one point, the two classic theorems
extend in any metric space. We study the metric spaces induced by simple graphs where no
line is a proper subset of another, and show that the least number of lines for such a graph with $n$
vertices is between the order of $n^{4/3}$ and $n^{4/3} \ln^{2/3} n$.
\end{abstract}


The classic Sylvester-Gallai theorem~(\cite{S},~\cite{G}) states that, for any $n$ points in the
plane, either all these points are collinear, or there is a line passing through
only two of them. Another classic theorem of De Bruijn and Erd\H{o}s~\cite{DBE} states that
for any set $V$ (called {\em points}) and subsets of $V$ (called {\em lines}), if any two points
is contained in exactly one line, then either there is only one line, or the number of lines
is no less than the number of points. As mentioned in \cite{DBE}, the De Bruijn-Erd\H{o}s theorem,
when restricted to the points and lines in the plane,
can be deduced from Sylvester-Gallai theorem by an easy induction.

V. Chv\'atal first investigated the possible generalizations of these theorems in arbitrary
metric spaces and then hypergraphs. In \cite{Chvatal} {\em lines} in metric space
are defined. Roughly speaking, a line $\overline{ab}$ contains $a$, $b$,
and any points $c$ where one of the triangle inequalities over $\{a, b, c\}$ is actually an equality.
We will give the formal definition in Section \ref{sect.general}.
It was observed that with such an definition, Sylvester-Gallai theorem does not extend to arbitrary metric
spaces. Then a new type of lines, which we will call {\em closure lines} are defined and with this
definition, the Sylvester-Gallai theorem generalizes in any metric space (\cite{Chvatal}, \cite{Chen}).

On the other hand, De Bruijn-Erd\H{o}s theorem does not generalize to metric spaces with the closure lines.
It is an open question (the Chen-Chv\'atal Conjecture in \cite{CC}) whether it generalizes
if we use lines instead of closure lines.

In this work, we actually take one step back. When one first encounters the definition of lines
in hypergraphs or metric spaces, it is natural to feel something strange. Two prominent problems
one may observe are
\begin{itemize}
\item One line might be the proper subset of another line.
\item Two lines might intersect at more than one point.
\end{itemize}

We are going to study the systems where one or both of the abnormalities do not happen.
We call a system {\em geometric dominant} if no line is a proper subset of another;
call a system {\em strongly geometric dominant} if any two lines intersect in at most one point.
In Section \ref{sect.general}, we give formal definitions and some basic facts. We
observe that strongly geometric dominant systems has several properties that conform
to geometric intuitions very well; and we note that the classic
theorems of Sylvester-Gallai and of De Bruijn-Erd\H{o}s both extend to any strongly geometric dominant
metric space. In Section \ref{sect.sgd} we characterize the strongly geometric dominant graphs.
Complete graph, path, and 4-cycle are easy examples of strongly geometric dominant graphs;
and we show that there are no others.
The second part of this article is the study on geometric dominant graphs that are not strong.
While small examples are rare, we show in Section \ref{sect.gd} that geometric dominant graphs
are abundant, even for graphs with stronger restrictions, which we call {\em super geometric dominant}.
In Section \ref{sect.gd_ub} we use the super geometric dominant graphs to construct
non-trivial geometric dominant graphs with as few as $O(n^{4/3} \ln^{2/3} n)$ lines. In Section \ref{sect.gd_lb} we prove that
any non-trivial geometric dominant graph has at least $\Omega(n^{4/3})$ lines. Thus we prove the Chen-Chv\'atal conjecture (\cite{CC})
in this special case for large $n$, and give almost tight bound on the least number of lines.

\section{Definitions and general observations}\label{sect.general}

A {\em hypergraph\/} is an ordered pair $(V,H)$ such that $V$ is a finite set
and $H \subseteq 2^V$ is a family of subsets of $V$; elements of $V$ are the {\em
  vertices\/} of the hypergraph and members of $H$ are its {\em
  hyperedges;\/} a hypergraph is called {\em $k$-uniform\/} if each of
its hyperedges consists of $k$ vertices; i.e. $H \subseteq \binom{V}{k}$.

The definition of lines and closure lines were first considered by Chv\'atal in metric spaces,
and generalized to hypergraphs in \cite{CC}.
Given any 3-uniform hypergraph, for any $u$, $v \in V$, the {\em line} $\ov{uv}$ is defined as
\[
\ov{uv}\;=\;\{u,v\}\cup\{p:\{u,v,p\}\in H\}.
\]
Unless otherwise specified in this work, all the hypergraph we consider are 3-uniform.
Because we only focus on ternary relations, any bigger hyperedge can be views as the collection
of all its subsets of size $3$.

A line is called {\em universal} if it contains all the vertices.
Three distinct vertices $a$, $b$, and $c$ are called {\em collinear}
if $\{a, b, c\} \in H$.

The {\em closure line} $\widetilde{uv}$ is defined as the transitive closure of $\{u, v\}$
with respect to $H$, where $H$ is viewed as a ternary relation over $V$. i.e., we first
take the vertices in line $\ov{uv}$, then keep taking vertices in line $\ov{ab}$
whenever $a$ and $b$ are included in our line, until no new vertices can be taken.

In a metric space $(V, \rho)$,
the natural ternary relation gives a hypergraph

\[ H_\rho = \{\{a, b, c\} : \rho(a, b) + \rho(b, c) = \rho(a, c) \}. \]

Any connected weighted graph with positive weights induce a metric space where the distance of two vertices
is defined as the length of a shortest path between them. In fact any finite metric space is trivially
induced by one such graph. We also study the metric spaces induced by unweighted connected graphs, where the
shortest path is simply the least number of steps between two vertices.

While the definition of lines is more natural than that of the closure lines, it was noted that
the Sylvester-Gallai theorem no longer holds in arbitrary metric spaces with the lines thus defined,
but it holds in a sense with the closure lines:

\begin{thm}\label{thm.sylvester_chvatal} (\cite{Chen})
In any metric space, either there is a universal closure line consisting of all the points,
or else there is a closure line of size 2.
\end{thm}

For the De Bruijn-Erd\H{o}s theorem, the story is quite different. One easily observes that
there are arbitrary big metric spaces where no closure line includes all the points, yet the number of
closure lines is a constant. However, in terms of lines defined as above,
the following Chen-Chv\'atal Conjecture remains open.

\begin{conj}\label{conj.cc} (\cite{CC}) In any finite metric space $(V, \rho)$, either $|V|$ is a line,
or else there are at least $|V|$ distinct lines.
\end{conj}

Let $g(n)$ be the least number of lines in a system (hypergraph, metric space, graph, maybe with restrictions, depending on the context)
with $n$ points, under the assumption that there is no universal line.
Conjecture \ref{conj.cc} states that $g(n)$ is at least $n$ for any metric spaces.
In \cite{CC} it was showed for hypergraphs in general $g(n)$ can be as few as $\exp(O(\sqrt{\ln n}))$, but can be no less than $\log_2 n$.
Recently the lower bound was improved to $(2 - o(1))\log_2 n$ in \cite{ABCCCM}.
Special cases of Conjecture \ref{conj.cc} were proved.
For example, $g(n) \geq n$ (\cite{Chvatal2}) and in fact $g(n) \in \Theta(n^{4/3})$ (\cite{Chini_Chvatal})
for metric spaces where all the distances belong to $\{0, 1, 2\}$.
Kantor and Patk\'os (\cite{KP}) showed a linear lower bound of $g(n)$ for metric spaces
induced by points in the 2-dimensional plane with $L_1$ distance.
We refers to \cite{ABCCCM} for a more detailed survey of related results.

In this article, we denote, for distinct points $a_0$, $a_1$, ..., $a_k$,
\[ [a_0 a_1 ... a_k] := \rho(a_0, a_k) = \sum_{i=0}^{k-1} \rho(a_i, a_{i+1}) \]
With this notation, for a metric space, three distinct points $a$, $b$, and $c$ are
collinear if $[acb]$ or $[cab]$ or $[abc]$, and the line $\ov{ab}$ contains $a$, $b$, and any $c$
that is collinear with $a$ and $b$. In particular, for lines in (the metric space induced by)
graphs, such $c$ means that one of $\{a, b, c\}$ lie on a shortest path between the other two.

The following facts are obvious. We list them and will use them frequently.

\begin{fact} In any metric space $(V, \rho)$, and for distinct points $a$, $b$, $c$, $d$, $a_i$ ($i = 0, 1, ..., k$),

(a) $[abc] \Leftrightarrow [cba]$;

(b) $[abc]$ and $[acb]$ cannot both hold;

(c) $[abc]$ and $[acd]$ implies $[abcd]$;

(d) $[a_0a_1...a_k]$ implies $\rho(a_s, a_t) = \sum_{i=s}^{t-1} \rho(a_i, a_{i+1})$ for any $s < t$.
\end{fact}

For two vertices $a$ and $b$ in a graph $(V, E)$, we denote $a \sim b$
(or $a \sim_G b$ when we want to emphasize the underlying graph) if $a$ and $b$ are adjacent, otherwise $a \not\sim b$.
We simply write $ab$ to denote the distance between $a$ and $b$ in the graph.
For any vertex $a$, $N(a) = \{b \in V : a \sim b\}$ is the {\em neighborhood} of $a$;
the degree of $a$ is $\deg(a) = |N(a)|$.
And $N^*(a) = N(a) \cup \{a\}$ is the {\em closed neighborhood} if $a$.
Two vertices $a$ and $b$ are called (non-adjacent) {\em twins} if they have the same neighborhood $N(a) = N(b)$.
Note that this implicitly implies $a \not\sim b$.
We refer to the standard textbook \cite{BM} for symbols and terms that are not defined in details here.

\begin{defi} A hypergraph $(V, H)$ is {\em geometric dominant} if no line is a proper subset of another.
\end{defi}

\begin{defi} A hypergraph $(V, H)$ is {\em strongly geometric dominant} if the intersection of any two lines has at most one element.
\end{defi}

We call a metric space (strongly) geometric dominant if its induced hypergraph is (strongly) geometric dominant. Similarly,
we have a (strongly) geometric dominant graph if the metric space induced by the graph has the corresponding property.

Because each line has at least 2 points, we have

\begin{fact} If a hypergraph is strongly geometric dominant, then it is geometric dominant.
\end{fact}

\begin{eg} Consider the metric space induce by the wheel with one center and 5 other vertices. It is geometric dominant but not strongly geometric dominant.
There are 15 different lines, each of size 4.
\end{eg}

We show that strongly geometric dominant systems has properties that conform to those of points and lines
in plane geometry.
Being strongly geometric dominant is equivalent to the condition that any line is generated by any two points inside it;
and this in turn is exactly the same as requiring that lines and closure lines coincide.

\begin{lem}\label{lem.leqcl} In a hypergraph $(V, H)$, if $\ov{uv} = \widetilde{uv}$ for any $u, v \in V$, then for any hyperedge
$\{a, b, c\} \in H$,
$\ov{ab} = \ov{ac}$.
\end{lem}

\begin{proof} $\{a, b, c\} \in H$, so $a$, $b \in \ov{ac}$. By the definition of closure lines, $\ov{ab} \subseteq \widetilde{ac}$. Since $\widetilde{ac} = \ov{ac}$, so $\ov{ab} \subseteq \ov{ac}$. Similarly, $\ov{ac} \subseteq \ov{ab}$.
\end{proof}

\begin{fact}\label{fact.leqcl} For any hypergraph $(V, H)$, the following are equivalent.

(a) $(V, H)$ is strongly geometric dominant.

(b) $\ov{uv} = \widetilde{uv}$ for any $u \neq v \in V$.

(c) For any line $L$ and any $u \neq v \in L$, $\ov{uv} = L$.
\end{fact}

\begin{proof} (a) $\Rightarrow$ (b): Let $a, b \in \ov{uv}$, and any $c$ such that $\{a, b, c\} \in H$. We have $c \in \ov{uv}$,
i.e., $\ov{uv}$ is already closed with respect to $H$.
Otherwise $\ov{uv}$ and $\ov{ab}$ are two different lines that intersect at both $a$ and $b$.

(b) $\Rightarrow$ (c): Consider a line $L = \ov{ab}$ and any other pair $u \neq v \in L$ such that $\{u, v\} \neq \{a, b\}$.
If $| \{u, v\} \cap \{a, b\} | = 1$, we may assume $v = a$ and use Lemma \ref{lem.leqcl},
\[ u \in \ov{ab} \Rightarrow \{a, b, u\} \in H \Rightarrow \ov{au} = \ov{ab} = L.\]
Otherwise, $\{u, v\}$ and $\{a, b\}$ are disjoint, using the same argument twice we have $\ov{ua} = \ov{ab} = L$, then
$\ov{uv} = \ov{ua} = L$.

(c) $\Rightarrow$ (a): Let $L_1$ and $L_2$ be two lines intersect at at least two distinct vertices $u$ and $v$, we have $L_1 = L_2 = \ov{uv}$.
\end{proof}

Now it is a simple fact to note that
Sylvester-Gallai theorem extends to any strongly geometric dominant metric space.

\begin{thm} In any strongly geometric dominant metric space $(V, \rho)$, either $V$ is the only line, or else there is a line of size 2.
\end{thm}

\begin{proof}
By Theorem \ref{thm.sylvester_chvatal} and Fact \ref{fact.leqcl}, if the space is strongly geometric dominant, either $V$ is a line or else there is a line of size 2. In the former
case, since the space is geometric dominant, there can be no other lines.
\end{proof}

We reformulate the original De Bruijn-Erd\H{o}s theorem in our setting.

\begin{thm} In any strongly geometric dominant hypergraph $(V, H)$, either $V$ is the only line, or there are at least $|V|$ lines.
\end{thm}

\begin{proof} For any two vertices $a$ and $b$, they are in $\ov{ab}$ and, because the hypergraph is strongly geometric, in no other lines.
  So any two vertices are contained in exactly one line. The original De Bruijn-Erd\H{o}s theorem applies.
\end{proof}

\section{Strongly geometric dominant graphs}\label{sect.sgd}

It is easy to see that the path $P_n$, the complete graph $K_n$, and the cycle $C_4$ are strongly geometric dominant.

\begin{lem}\label{lem.v_line} If $G = (V, E)$ is a geometric dominant graph where $V$ is a universal line, then $G$ is a path or $C_4$.
\end{lem}

\begin{proof} Since $V$ is a line and $G$ is geometric dominant, we have
\begin{equation}\label{eq.v_line}
V \text{ is the only line in } G.
\end{equation}
So there cannot be any triangles in $G$, otherwise any triangle $abc$ will have $a \not\in \ov{bc}$ which  contradicts (\ref{eq.v_line}).
We further prove that there are no vertices of degree $3$ or more. Suppose $a$ has neighbors $b$, $c$, and $d$. Because there are no triangles,
$b$, $c$, and $d$ are pairwise non-adjacent, so $bc=cd=db = 2$, and this implies $b \not\in \ov{cd}$ which contradicts (\ref{eq.v_line}) as well.

So the graph, being connected, is a path or a cycle. It is easy to check $C_4$ is the only cycle satisfying (\ref{eq.v_line}).
\end{proof}

\begin{defi} A graph is called {\em non-trivial} if it is not a complete graph, nor a path, nor $C_4$.
\end{defi}

\begin{lem}\label{lem.bridge} In a non-trivial geometric dominant graph, if three distinct vertices $a$, $b$, and $c$ satisfy $a \sim b$, $b \sim c$ and $a \not\sim c$,
then there exists another vertex $d$ such that $d$ is adjacent to all
of $a$, $b$, and $c$.
\end{lem}

\begin{proof}
We first prove that $b$ has neighbors other than $a$ and $c$.
Otherwise, for any other vertex $z \in \ov{ac}$, we have (1) $[zabc]$ or (2) $[zcba]$ or (3) $az = zc = 1$, $bz = 2$. In any case, $z \in \ov{ab}$ and $z \in \ov{bc}$. So $\ov{ac} \subseteq \ov{ab}$ and $\ov{ac} \subseteq \ov{bc}$. Since the graph is geometric dominant,
$\ov{ab} = \ov{ac} = \ov{bc}$. Now consider any vertex $x$ and a shortest path between $x$ and $b$. Since $a$ and $c$ are $b$'s only neighbors,
the path goes through one of $a$ and $c$, so $x \in \ov{ab}$ or $x \in \ov{bc}$. But since $\ov{ab} = \ov{bc} = \ov{ac}$, so $x \in \ov{ac}$. This implies
that $\ov{ac} = V$ and $G$ is trivial by Lemma \ref{lem.v_line}.

Now, pick any vertex $z \in N(b) \setminus \{a, c\}$, if $z \sim a$ and $z \sim c$, then we set $d = z$ and done. Without loss of generality, $z \not\sim a$.
So $z \in \ov{ab}$ and $z \not\in \ov{ac}$ (since $az = ac = 2$ and $zc \in \{1, 2\}$). Because the graph is geometric dominant, we have to find a vertex which lies on
$\ov{ac}$ but not $\ov{ab}$.

Consider lines $\ov{ac}$ and $\ov{ab}$, any vertex $x$ on $\ov{ac}$ satisfying $[xac]$ will have $[xabc]$ therefore  $x \in \ov{ab}$.
Similarly, $[acx]$ implies $x \in \ov{ab}$. So there must be a vertex $x$ such that $[axc]$ and $x \not \in \ov{ab}$. So $x \sim a$, $x \sim c$,
and $x \sim b$ (otherwise $x \in \ov{ab}$). Now let $d = x$ and we are done.
\end{proof}

\begin{thm} A connected graph is strongly geometric dominant if and only if it is a path or a complete graph or $C_4$.
\end{thm}

\begin{proof}
The ``if'' part is easy to check. We prove the other direction. If the graph is not trivial (in particular not complete),
there are three distinct vertices $a \sim b \sim c$ and $a \not\sim c$, and by Lemma \ref{lem.bridge},
another vertex $d$ adjacent to all of $\{a, b, c\}$.
Now $\ov{ab}$ and $\ov{ac}$ are two different lines ($d \in \ov{ac} - \ov{ab}$)
that intersect in more than one vertices ($\{a, b, c\} \subseteq \ov{ab} \cap \ov{ac}$),
contradicts the assumption that the graph is strongly geometric dominant.
\end{proof}

\section{Which graphs are geometric dominant?}\label{sect.gd}

The smallest non-trivial example is the wheel with 6 vertices. Then on 7 vertices, there is one (up to isomorphism) example,
which is the wheel plus a twin of a non-center vertex. By adding one vertex (also a twin of a non-center vertex) to the graph of order 7,
we have one example of geometric dominant graph of order 8. We do not know whether there are such graph of order 9.

On the other hand, while small non-trivial geometric dominant graphs are hard to find, when the number of vertices is big, the geometric dominant graphs are abundant.

\begin{fact} For each $n \geq 16$, there is a non-trivial geometric dominant graph of order $n$.
\end{fact}

\begin{proof} Start from the 5-wheel. We substitute a stable set of size at least 3 for each non-center vertex. It is easy to check that the resulting graph is geometric dominant.
\end{proof}

\begin{defi}\label{defi.general_graph} A connected graph $G = (V, E)$ is called {\em super geometric dominant} if it satisfies the following
\begin{enumerate}
\item The diameter of $G$ is 2.
\item For any $a$, $b$, $c$, $d \in V$, where $a \neq b$, $c \neq d$, and $\{a, b\} \neq \{c, d\}$, we have $\ov{ab} \not\subseteq \ov{cd}$.
\item For any $a \neq b \in V$, we have $N^*(a) \not\subseteq N^*(b)$ where $N^*(x)$ is
the neighbors of $x$ plus $x$ itself.
\item For any $a$, $b$, $c \in V$, where $b \neq c$, we have $N^*(a) \not\subseteq \ov{bc}$ and $\ov{bc} \not\subseteq N^*(a)$.
\end{enumerate}
\end{defi}

In words, besides having diameter 2, we take the family of all the $\binom{n}{2}$ lines and $n$ closed neighborhoods, we require that all these sets to form an antichain. Clearly such graphs are geometric dominant. The extra requirements will be helpful when we construct geometric dominant graphs having few lines. The next theorem states that the random graph is almost surely not only geometric dominant, but also super.

\begin{thm}\label{thm.super} When $n \rightarrow \infty$ and $p$ is a function such that
  \[p \in \omega\left(\sqrt[3]{\frac{\ln n}{n}} \right) \;\; \text{ and } \;\; 1 - p \in \omega\left(\sqrt{\frac{\ln n}{n}} \right), \]
the random graph $\mathcal{G}_{n, p}$ is super geometric dominant (therefore geometric dominant) almost surely.
\end{thm}

We put the proof of Theorem \ref{thm.super} in the appendix. It is quite similar to the proof of Theorem \ref{thm.super_super}, which
is more important in the next section.
By taking $1 - p = C \sqrt{\ln n / n}$ in Theorem \ref{thm.super}, we see the existence of super geometric dominant graphs that missed
only $O(n^{3/2} \sqrt{\ln n})$ edges. By using a less symmetric construction, we show that there are graphs missing even much
fewer edges.

\begin{thm}\label{thm.super_super} There is a family of graphs $\{G_n\}$, a constant $C > 0$, and a constant $N > 0$, such that, for $n > N$, $G_n$ is a super geometric dominant graph on
$n$ vertices such that the complement of $G_n$ has less than $C n \ln n$ edges.
\end{thm}

\begin{proof} Let $V$ be a set of $n$ vertices and partition it into the left side $L$ and the right side $R$ with $|L| = t$ and $|R| = n-t$,
  where $t$ will be specified later. And we consider the random construction where $R$ is a clique, and we take any other edge with
  probability 1/2 independently. We are going to show that, with appropriate values of $t$ and $p$, such a graph is super geometric dominant with positive probability.
  We call a three-vertex set {\em tight} if they induce two edges.

Note that sometimes we talk about events with probability 1, such as $a \sim b$ when $\{a, b\} \subseteq R$. One may think we
chose any edge inside $R$ with probability 1. When in some cases we want to condition on some events with probability 0,
we note that the contribution is just 0 to the total probability and will not affect our bound.
We define and bound (the probability of the complement of) the following events.

For any two vertices $a$ and $b$, $D_{ab}$ is the event that there is $z \in L - \{a, b\}$, such that $a \sim z$ and $b \sim z$. And let $D$ be the intersection of all the $D_{ab}$'s, so $D$ implies that the graph has diameter at most 2. Note that there are at least $t - 2$ vertices in $L$ that differ from $a$ and $b$, so
\[\prob(\ov{D_{ab}}) \leq (1 - 1/4)^{t-2} \leq \exp(-(t-2) / 4).\]

For any two vertices $a$ and $b$, $N_{ab}$ is the event that there is $z \in L -  \{a, b\}$, such that $a \sim z$ and $b \not\sim z$.
\[\prob(\ov{N_{ab}}) \leq (1 - 1/4)^{t-2} \leq \exp(-(t-2) / 4).\]

For any three vertices $a$, $b$, and $c$, $L_{abc}$ is the event that there is $z \in L - \{a, b, c\}$, such that $\{a, b, z\}$ is tight and $\{a, c, z\}$ is not tight. Note that any outcome in $D \cap L_{abc}$ has the property that lines $\ov{ab} \not\subseteq \ov{ac}$.
Condition on any event whether $a \sim b$ and whether $a \sim c$, if we consider any $z \in L - \{a,, b, c\}$
and whether it is adjacent to $a$, $b$, and $c$, at least one of the 8 outcomes will make $\{a, b, z\}$ tight and $\{a, c, z\}$ not ---
when $a \sim b$, we want $N(z) \cap \{a, b, c\} = \{b\}$;
when $a \not\sim b$ and $a \sim c$, we want $\{a, b, c\} \subseteq N(z)$;
and when $a \not\sim b$ and $a \not\sim c$, we want $N(z) \cap \{a, b, c\} = \{a, b\}$;
So,
\[\prob(\ov{L_{abc}}) \leq (1 - 1/8)^{t-3} \leq \exp(-(t-3) / 8).\]

For any four distinct vertices $a$, $b$, $c$, and $d$, $L_{abcd}$ is the event that there is $z \in L - \{a, b, c, d\}$, such that
$\{a, b, z\}$ is tight and $\{c, d, z\}$ is not tight. Again, condition on any event whether $a \sim b$ and whether $c \sim d$, for any
$z \in L - \{a, b, c, d\}$, we consider its adjacency relations with $\{a, b, c, d\}$. Clearly, at least one of the 16 outcomes makes $\{a, b, z\}$ tight and $\{c, d, z\}$ not. So
\[\prob(\ov{L_{abcd}}) \leq (1 - 1/16)^{t-4} \leq \exp(-(t-4) / 16).\]

For any two vertices $a$ and $b$, $E_{ab}$ is the event that there is $z \in L - \{a, b\}$, such that $z \sim a$ and $\{a, b, z\}$ is not tight. Condition on $a \sim b$ or not, we just need $z \sim a$, and $zb$ has the same adjacency relation as $ab$. This gives us the bound
\[\prob(\ov{E_{ab}}) \leq (1 - 1/4)^{t-2} \leq \exp(-(t-2) / 4).\]

For any two vertices $a$ and $b$, $E'_{ab}$ is the event that either $a \not\sim b$, or there is $z \in L -  \{a, b\}$,
such that $z \not\sim a$ and $z \sim b$. Note that any outcome in $D \cap E'_{ab}$ has the property that $\ov{ab} \not\subseteq N^*(a)$.
\[\prob(\ov{E'_{ab}}) = 0 + \prob(\ov{E'_{ab}} | a \sim b) \prob(a \sim b)  \leq (1 - 1/4)^{t-2} \leq \exp(-(t-2) / 4).\]

For any three vertices $a$, $b$, and $c$, $E_{abc}$ is the event that there is $z \in L - \{a, b, c\}$
such that $z \sim a$ and $\{b, c, z\}$ is not tight; and $E'_{abc}$ is the event that there is $z \in L - \{a, b, c\}$
such that $z \not\sim a$ and $\{b, c, z\}$ is tight. Condition on any event whether $b \sim c$ or not, for each $z \in L - \{a, b, c\}$,
consider the possible edges $za$, $zb$, $zc$. Clearly, at least one of the 8 outcomes will make $E_{abc}$ happen and another makes $E'_{abc}$
happen, so both $\prob(\ov{E_{abc}})$ and $\prob(\ov{E_{abc}})$ are bounded above by $(1 - 1/8)^{t-3} \leq \exp(-(t-3) / 8)$.

Thus we defined $O(n^4)$ events and the probability of the complement of each is bounded above by $\exp(-(t-4)/16)$. We pick $t = C_0 \ln n$
for big enough $C_0$, so that the intersection of all the events has positive probability.
Note that any outcome in the intersection of all the events is a super geometric dominant graph
where the number of edges in the complement of the graph is at most $\binom{t}{2} + t(n-t) = O(n \ln n)$.

\end{proof}

\section{Lines in geometric dominant graphs --- upper bound}\label{sect.gd_ub}

\begin{defi} For each $n$, define $g(n)$ to be the least number of lines of a non-trivial geometric dominant graph on $n$ vertices.
\end{defi}

In this and the next section we give lower and upper bounds for $g(n)$.

\begin{defi} Let $G = (V, E)$ be a graph, $t > 0$, the {\em $t$-exploded graph of $G$} is defined to be the graph
$G[t]$ where the vertex set is the union of disjoint sets $V_v$'s for each $v \in V(G)$, for each $u \sim_G v$, $(V_u, V_v)$
form a complete bipartite graph in $G[t]$, and $G[t]$ has no other edges.
\end{defi}

i.e., $G[t]$ is the graph constructed from $G$ by substitute each vertex with a stable set of size $t$.
The following fact is the reason why we needed the extra requirements for super geometric dominant graphs,
as well as why we were aiming for super geometric dominant graphs which miss as few edges as possible.
The validity of the fact is easy to check.

\begin{fact}\label{fact.explode} If $G$ is a super geometric dominant graph with $n$ vertices and $m$ edges, and $t \geq 3$.
In the exploded graph $H = G[t]$, the diameter is also 2, and we have the lines
\begin{enumerate}
\item For each $v \in V$ in $G$ and $a, b \in V_v$ in $H$, $\ov{ab}_{(H)} = \{a, b\} \cup \bigcup_{v \sim c} V_c$.
\item For $u \not\sim v \in V$ in $G$, $a \in V_u$ and $b \in V_v$ in $H$,
\[ \ov{ab}_{(H)} = \{a, b\} \cup \bigcup_{c \in \ov{uv}, c \not\in \{u, v\}} V_c.\]
\item For $u \sim v \in V$ in $G$, $a \in V_u$ and $b \in V_v$ in $H$,
\[ \ov{ab}_{(H)} = \bigcup_{c \in \ov{uv}} V_c.\]
\end{enumerate}
$H$ is geometric dominant with $nt$ vertices and the number of lines is
\[ \binom{t}{2} n + \left( \binom{n}{2} - m \right) t^2 + m,\]
where each term in the sum corresponding to the three types of lines above.
\end{fact}

\begin{thm} $g(n) \in O(n^{4/3} \ln^{2/3} n)$.
\end{thm}

\begin{proof} Let $n_0 = \lfloor n^{2/3} \ln^{1/3} n \rfloor$ and $t = \lceil n / n_0 \rceil \in O((n / \ln n)^{1/3})$.
By Theorem \ref{thm.super_super}, there exists a super geometric dominant graph $G$ with less than $O(n_0 \ln n_0)$
missing edges. We explode $G$ to $G[t]$ and delete vertices to make the number of total vertices $n$,
while keeping the parts as balanced as possible. By Fact \ref{fact.explode}, the number of lines is bounded by
\[O(t^2 n_0 + t^2 n_0 \ln n_0 + n_0^2) = O(n^{4/3} \ln^{2/3} n) \]
\end{proof}

\section{Lines in geometric dominant graphs --- lower bound}\label{sect.gd_lb}

\begin{lem}\label{lem.abc} In a non-trivial geometric dominant graph, if $\ov{ab} = \ov{ac}$ for $b \neq c$, then $[bac]$.
\end{lem}

\begin{proof} Otherwise, without loss of generality, $[abc]$. Let $d$ be the vertex before $b$ and $e$ be the vertex after $b$ on a shortest path from $a$ to $c$ that go through $b$. (It is possible $a = d$ or $e = c$.) By Lemma \ref{lem.bridge}, there is $b'$ adjacent to $d$, $b$, $e$. $ab' \leq ad + 1 = ab$. But if $ab' < ab$, we have
\[ac \leq ab' + b'e + ec = ab' + 1 + ec = ab' + bc < ab + bc = ac,\]
which is impossible. Therefore, $ab' = ab$. Similarly, $cb' = cb$. This implies $b' \in \ov{ac}$. However, $ab' = ab$ and $b \sim b'$ implies $b' \not\in \ov{ab}$. A contradiction.
\end{proof}

Using Lemma \ref{lem.bridge} in the similarly way, it is easy to check the following.

\begin{lem}\label{lem.abcd} If $a$, $b$, $c$, $d$ are four distinct vertices (in any order) on a shortest path in a non-trivial geometric dominant graph, then $\ov{ab} \neq \ov{cd}$.
\end{lem}

\begin{lem}\label{lem.xyz} If $a$, $b$, and $c$ are three points in a metric space where the distances are $ab = x+y$, $bc = y+z$, and $ca=z+x$ for some positive values of $x$, $y$, and $z$, then $a$, $b$, $c$ are not collinear.
\end{lem}

\begin{proof} It is easy to check any of the three distances is strictly less than the sum of the other two.
\end{proof}

\begin{lem}\label{lem.star_generator} In a non-trivial geometric dominant graph, a vertex $a$ and $t$ other vertices $B = \{b_i : i = 1, ..., t\}$. If $\ov{a b_i}$ are all the same, then $\ov{b_i b_j} \cap B = \{b_i, b_j\}$ for any $i \neq j$.
\end{lem}

\begin{proof} Let $b_i$, $b_j$, and $b_k$ be 3 distinct vertices in $B$. By Lemma \ref{lem.abc}, $b_ib_j = ab_i + ab_j$, $b_jb_k = ab_j + ab_k$, $b_kb_i = ab_k + ab_i$, then Lemma \ref{lem.xyz} implies that $b_k \not\in \ov{b_ib_j}$.
\end{proof}

\begin{lem}\label{lem.parallel} Let $a$, $b$, $c$, and $d$ be four distinct vertices in a non-trivial geometric dominant graph, and $\ov{ab} = \ov{cd}$. Then we have $ab = cd$, $ac = bd$ and $ad = bc$; furthermore, if $ab > 1$, then $[acb]$, $[adb]$, $[cad]$, and $[cbd]$.
\end{lem}

\begin{proof}

{\em Case 1:} None of $[acb]$, $[adb]$, $[cad]$, and $[cbd]$ holds.
Among the distances $ac$, $bd$, $ad$, and $bc$, we may assume that $ac$ is (one of) the smallest. Because $\{a, c\} \subseteq \ov{ab} = \ov{cd}$, we have, by the minimality of $ac$,
\begin{equation}\label{eq.bc_ad}
[bac] \text{ i.e. } bc = ab + ac, \;\text{ and }\; [acd] \text{ i.e. } ad = ac + cd.
\end{equation}
Because $d \in \ov{ab}$, we have $[abd]$ or $[dab]$ (we assumed $[adb]$ does not happen in this case).
$[dab]$ and $[acd]$ in (\ref{eq.bc_ad}) imply $[dcab]$, and in turn  $\ov{ab} \neq \ov{cd}$ by Lemma \ref{lem.abcd}.
So we must have
\begin{equation}\label{eq.d_from_ab}
[abd], \text{ i.e. } bd = ad - ab = ac + cd - ab.
\end{equation}
Because $b \in \ov{cd}$, we have $[bcd]$ or $[cdb]$ (we assumed $[cbd]$ does not happen in this case).
$[bcd]$ and $[bac]$ in (\ref{eq.bc_ad}) imply $[bacd]$, and in turn  $\ov{ab} \neq \ov{cd}$ by Lemma \ref{lem.abcd}.
So $[cdb]$ and $bd = bc - cd = ac + ab - cd$. Together with (\ref{eq.d_from_ab}) and (\ref{eq.bc_ad}), we get $ab = cd$, $ac = bd$, and $ad = bc$. Furthermore, we have $ab = 1$ in this case. Otherwise, pick one vertex $z$ such that $[azb]$, $az=1$ and $zb=ab-1$.
By (\ref{eq.bc_ad}), $[cazb]$ and $cz = ca + 1$.
Also by (\ref{eq.bc_ad}) $da = ac + cd = db + ba$, so $[dbza]$ and $dz = db + bz = ac + (ab - 1)$. $cd = ab = (ab - 1) + 1$.
Then by Lemma \ref{lem.xyz} we have $z \not\in \ov{cd}$. Yet $z \in \ov{ab}$, a contradiction.

{\em Case 2:} Some of $[acb]$, $[adb]$, $[cad]$, and $[cbd]$ holds. We may assume
\begin{equation}\label{eq.abcd_2_1}
[acb] \text{ i.e. } ab = ac +  bc
\end{equation}
$[dab]$ or $[dba]$ implies $[dacb]$ or $[dbca]$, both contradict the fact $\ov{ab} = \ov{cd}$ by Lemma \ref{lem.abcd}. Because $d \in \ov{ab}$, we must have
\begin{equation}\label{eq.abcd_2_2}
[adb] \text{ i.e. } ab = ad +  bd
\end{equation}
Now $[cda]$ and (\ref{eq.abcd_2_1}) imply $[bcda]$, $[dca]$ and (\ref{eq.abcd_2_2}) imply $[bdca]$. Both cases are impossible by Lemma \ref{lem.abcd}. Because $a \in \ov{cd}$, we get
\begin{equation}\label{eq.abcd_2_3}
[cad] \text{ i.e. } cd = ac +  ad
\end{equation}
Similarly,
\begin{equation}\label{eq.abcd_2_4}
[cbd] \text{ i.e. } cd = bc +  bd
\end{equation}
Equations (\ref{eq.abcd_2_1}) to (\ref{eq.abcd_2_4}) imply that $ab = cd$, $ac = bd$, and $ad = bc$.
\end{proof}

Note that when $a \sim b$, $ac$ and $bc$ differ by at most 1. We immediately have

\begin{lem}\label{lem.parity} In any connected graph, if $a \sim b$, then for any $c$, $c \in \ov{ab}$ if and only if $ac$ and $bc$ are different in parity. I.e.,
\[ c \in \ov{ab} \Leftrightarrow ac \oplus bc = 1 \text{ (in the binary field $\mathbf{F}_2$). } \]
\end{lem}

\begin{lem}\label{lem.unit_square} In a non-trivial geometric dominant graph, if $\ov{ab} = \ov{cd}$ for distinct vertices $a$, $b$, $c$, and $d$, and $ab = cd = ac = bd = 1$, then $\ov{ac} = \ov{bd}$.
\end{lem}

\begin{proof}
For any vertex $x$, by Lemma \ref{lem.parity} and the assumption that $\ov{ab} = \ov{cd}$,
\[ xa \oplus xb = xc \oplus xd \Rightarrow xa \oplus xc = xb \oplus xd ,\]
the latter implies that $x$ is in both $\ov{ac}$ and $\ov{bd}$, or none of them.
\end{proof}

\begin{defi} Let $L$ be a line in a graph $G = (V, E)$, the {\em generator graph of $L$ in $G$} is defined to be $H(L) = H_G(L) = (V, E_H)$ where $a \sim_{H(L)} b$ whenever $\ov{ab} = L$ in $G$.
\end{defi}

\begin{lem}\label{lem.bip_blocks} For a non-trivial geometric dominant graph $G$ and any line $L$, every connected component of $H(L)$ is a complete bipartite graph. Furthermore, if $H(L)$ is not a star, then

(a) There is a constant $d(L)$ such that for any $a \sim_{H(L)} b$ (i.e., $\ov{ab} = L$), their distance in $G$ satisfies $ab = d(L)$;
and $d(L) = 1$ unless $H(L)$ is a matching.

  (b) Call either side (of the vertex set) of a connected bipartite component {\em with at least two vertices} a block.
For any two blocks $X$ and $Y$, there is a constant $d(X, Y)$ such that for any $x \in X$ and $y \in Y$,
the distance in $G$ satisfies $xy = d(X, Y)$.
\end{lem}

\begin{proof} We first prove that $H(L)$ is bipartite. Assume there is an odd cycle $a_1 a_2 ... a_{2t+1}$, Lemma \ref{lem.abc} implies that $t > 1$. Successively apply Lemma \ref{lem.parallel},
we have the distances
\[ a_{t+1} a_{t+2} = a_{t}a_{t+3} = ... = a_2 a_{2t+1} .\]
But note that $\ov{a_{t+1} a_{t+2}} = \ov{a_1 a_2} = \ov{a_1 a_{2t+1}} = L$, by Lemma \ref{lem.parallel} and $t > 1$ we have
\[a_2 a_{2t+1} = a_{t+1} a_{t+2} = a_1 a_2 = a_1 a_{2t+1}.\]
So $a_{2t+1} \not\in \ov{a_1 a_2} = L$, a contradiction.

Next consider any connected component of $H$ that is not a star, pick any (when the component has at least two vertices on both sides, we always have) four distinct vertices in the component such that
\[ a \sim b \sim c \sim d \text{ in $H$} .\]
By Lemma \ref{lem.parallel}, we have $ab = cd$, $ac = bd$, and $ad = bc$ in $G$. By Lemma \ref{lem.abc}, we have $[abc]$ and $[bcd]$, i.e.,
\[ ac = ab + bc \text{ and } bd = bc + cd.\]
By Lemma \ref{lem.parallel} and the fact that $[acb]$ does not hold, we have $ab = cd = 1$.
We further prove that $bc = 1$. Otherwise, $ad = bc > 1$, we may pick a vertex $z$ such that $[azd]$, $az = 1$ and $zd = ad - 1$.
Note that
\[bd = bc + cd = ad + ab = ba + az + zd \Rightarrow [bazd] \Rightarrow bz = ba + az = 2.\]
\[ac = ab + bc = cd + ad = az + zd + dc \Rightarrow [azdc] \Rightarrow cz = cd + dz = 1 + (ad - 1). \]
Note that $bc = ad = (ad-1) + 1$. Lemma \ref{lem.xyz} implies that $z \not\in \ov{bc}$, yet $[bazd]$ implies $z \in \ov{ab} = \ov{bc}$.
A contradiction. This proves that

{\em Claim 1.} Whenever $a \sim b \sim c \sim d$ for four distinct vertices in $H$, we have $a \sim b \sim c \sim d$ in $G$.

And by Lemma \ref{lem.parallel} and \ref{lem.unit_square}, we have $d \sim a$ in $G$ and then $d \sim a$ in $H$. This means there can be no vertices with distance 3 in $H$, therefore
every component is a complete bipartite graph.

(a) If there is only one connected component in $H$ and it is not
a star, we proved it is complete bipartite, and by Claim 1, $ab = 1$ in $G$ for all $a \sim_{H(L)} b$.

If there are at least two components, by applying Lemma \ref{lem.parallel} we get all the distances $ab$'s for $a \sim_{H(L)} b$ are the same.
We still need to show that when the common distance is bigger than 1, $H$ must be a matching. Otherwise, we have distinct vertices $a$, $b$, $c$, $d$, and $e$ satisfying $\ov{ab} = \ov{ac} = \ov{de}$
and $ab = ac = de > 1$. Lemma \ref{lem.abc} implies that $[bac]$, Lemma \ref{lem.parallel} and $ab = ac >1$ imply that $[bda]$ and $[cda]$,
so $bc \leq bd + cd < ba + ca = bc$, a contradiction.

(b) Let $X$ and $Y$ be two blocks. Note that $|X|$ and $|Y|$ are both greater than 1. If $X$ and $Y$ are the two sides of the same component, (a) implies that $xy$ is 1 for any $x \in X$, $y \in Y$. Otherwise, suppose $(X, X')$, $(Y, Y')$ are two different complete bipartite connected components
of $H$. Let $x_1, x_2 \in X$, $x' \in X'$, $y_1, y_2 \in Y$, and $y' \in Y'$ (it is possible that $x_1 = x_2$ or $y_1 = y_2$). By Lemma \ref{lem.parallel} and note that
\[ \ov{x_i x'} = \ov{y_j y'} = L, i,j  \in \{1, 2\},\]
we have, in $G$,
\[ x_1 y_1 = x'y' = x_2 y_2 .\]

\end{proof}

\begin{lem}\label{lem.parallel_non_edge} In a geometric dominant graph, if there is a line $L$ and $2t$ distinct vertices
$a_i$, $b_i$ ($1 \leq i \leq t$) such that $a_i \not\sim b_i$ and $\ov{a_ib_i} = L$ for all $i$,
then there are at least  $\binom{t}{2}$ different lines in $G$.
\end{lem}

\begin{proof} By Lemma \ref{lem.parallel}, all the distances $a_ib_i$ equal to some constant $d(L) > 1$.
For each pair $i \neq j$,
pick any pair in $\{ \{a_i, a_j\}, \{b_i, b_j\}, \{a_i, b_j\}, \{b_i, a_j\} \}$ with the longest distance in $G$ among the four, and call it
the {\em original pair}, and let $L_{ij}$ be the line generated by the original pair. Note that when $d(L) > 2$,
the original pair must have distance bigger than 1.
We are going to prove that $L_{ij} \neq L_{kh}$ for any two different pairs $\{i, j\}$ and $\{k, h\}$.
Aiming a contradiction, we assume $L_{ij} = L_{hk}$. We have the following cases.

{\em Case 1:} $\{i, j\}$ and $\{k, h\}$ are not disjoint. We may assume $h = i$, that is $L_{ij} = L_{ik}$.

{\em Case 1.1:} The original pairs for $L_{ij}$ and $L_{ik}$ share one vertex. We may assume $L_{ij}=L_{ik}=\ov{a_ia_j} = \ov{a_ia_k}$
and $a_ia_j \geq a_ia_k$. Now
\begin{equation}\label{eq.parallel_non_edge_1_0}
\text{Lemma \ref{lem.abc} and } \ov{a_ia_j} = \ov{a_ia_k} \Rightarrow [a_ja_ia_k];
\end{equation}
\begin{equation}\label{eq.parallel_non_edge_1_1}
\text{Lemma \ref{lem.parallel} and } \ov{a_jb_j} = \ov{a_kb_k} \Rightarrow [a_ka_jb_k].
\end{equation}
\begin{equation}\label{eq.parallel_non_edge_1_2}
\text{(\ref{eq.parallel_non_edge_1_0}) and (\ref{eq.parallel_non_edge_1_1}) } \Rightarrow
[a_ka_ia_jb_k] \Rightarrow a_kb_k > a_ia_k + a_ia_j \geq 2 a_ia_k.
\end{equation}
On the other hand, by the definition of the original pairs, $a_ia_k \geq a_ib_k$, so
\begin{equation}\label{eq.parallel_non_edge_1_3}
\text{Lemma \ref{lem.parallel} and } \ov{a_ib_i} = \ov{a_kb_k} \Rightarrow [a_ka_ib_k] \Rightarrow a_kb_k = a_ia_k + a_ib_k \leq 2a_ia_k.
\end{equation}
(\ref{eq.parallel_non_edge_1_2}) and (\ref{eq.parallel_non_edge_1_3}) contradict.

{\em Case 1.2:} The original pairs for $L_{ij}$ and $L_{ik}$ are disjoint. We may assume $L_{ij} = L_{kh} = \ov{a_ia_j} = \ov{b_ib_k}$. By Lemma \ref{lem.parallel}, $a_ia_j = b_ib_k$.

{\em Case 1.2.1:} $a_ia_j = b_ib_k > 1$. By Lemma \ref{lem.parallel}, $[a_ib_ia_j]$. But $\ov{a_ib_i} = \ov{a_jb_j}$ and Lemma \ref{lem.parallel} imply that $[a_ia_jb_i]$, a contradiction.

{\em Case 1.2.2:} $a_ia_j = b_ib_k = 1$. By our choice of the original pairs, we have $d(L) = 2$ and
\begin{equation}\label{eq.parallel_non_edge_1_4}
a_ia_u = a_ib_u = b_ia_u = b_ib_u = 1, \text{ for $u = j, k$. }
\end{equation}
$\ov{a_ia_j} = \ov{b_ib_k}$ implies $a_j \in \ov{b_ib_k}$,
together with (\ref{eq.parallel_non_edge_1_4}) we get $a_jb_k = 2$.
$\ov{a_ib_i} = \ov{a_jb_j}$ implies $b_j \in \ov{a_ib_j} = \ov{b_ib_k}$,
together with (\ref{eq.parallel_non_edge_1_4}) we get $b_jb_k = 2$.
And $a_jb_j = d(L) = 2$.
The pairwise distances among $a_j$, $b_j$, $b_k$ are all 2, so $b_k \not\in \ov{a_jb_j}$,
contradicts the assumption that $\ov{a_jb_j} = \ov{a_kb_k}$.

{\em Case 2:} $i$, $j$, $k$, $h$ are four distinct indices. We may assume the original pairs are $\{a_i,a_j\}$ and $\{a_k,a_h\}$. By Lemma \ref{lem.parallel}, $a_ia_j = a_ka_h$.

{\em Case 2.1:} $a_ia_j = a_ka_h = 1$. This implies $a_k \in \ov{a_ia_j}$ and we may assume $[a_ia_ja_k]$.
$\ov{a_ib_i} = \ov{a_kb_k}$ implies that $[a_ia_kb_i]$, so we have $[a_ia_ja_kb_i]$.
This implies that $d(L) = a_ib_i > 2$ and, with $a_ia_j = 1$,
$\{a_i,a_j\}$ could not be the original pair for $L_{ij}$.

{\em Case 2.2:} $a_ia_j = a_ka_h > 1$. $\ov{a_ia_j} = \ov{a_ka_h}$, $\ov{a_ib_i} = \ov{a_jb_j}$ and Lemma \ref{lem.parallel}
imply that $[a_ia_ka_j]$ and $[a_ia_jb_i]$.
So $[a_ia_ka_jb_i]$ and we denote $a_ia_k = x$, $a_ka_j = y$, $a_jb_i = z$, and $a_ib_i = x + y + z$.
$\ov{a_ib_i} = \ov{a_jb_j} = \ov{a_kb_k}$ and Lemma \ref{lem.parallel} now imply
\begin{equation}\label{eq.parallel_non_edge_2_1}
\begin{gathered}
a_ia_k = b_ib_k = x, a_ka_j = b_kb_j = y, a_jb_i = a_ib_j = z, \\
a_ib_i = a_jb_j = a_kb_k = x + y + z.
\end{gathered}
\end{equation}
We note that
\begin{equation*}
\text{ (\ref{eq.parallel_non_edge_2_1}) } \Rightarrow
\begin{cases}
[a_ib_jb_kb_i] \Rightarrow a_ib_k = y + z \\
[a_jb_ib_kb_j] \Rightarrow a_jb_k = x + z \\
[a_ia_ka_jb_i] \Rightarrow a_ia_j = x + y.
\end{cases}
\end{equation*}
Now Lemma \ref{lem.xyz} implies that $b_k \not\in \ov{a_ia_j}$. But since $a_h \in \ov{a_kb_k}$, we have $b_k \in \ov{a_ka_h} = \ov{a_ia_j}$, a contradiction.

\end{proof}

\begin{lem}\label{lem.many_missing_edges} If the complement of a non-trivial geometric dominant graph $G$ has $m$ edges, then the number of lines in G is $\Omega(m^{2/3})$.
\end{lem}

\begin{proof} We partition the $m$ missing edges according to the lines their end points generate. If every part is smaller than $m^{1/3}$ then we are done. Otherwise, there is a line $L$ such that $H(L)$ has more than $m^{1/3}$ edges whose end points are not adjacent in $G$. If $H(L)$ is a star, Lemma \ref{lem.star_generator} implies $\Omega(m^{2/3})$ lines. Otherwise, Lemma \ref{lem.bip_blocks} implies that $H(L)$ is a matching, and Lemma \ref{lem.parallel_non_edge} implies that there are $\Omega(m^{2/3})$ lines.
\end{proof}

\begin{lem}\label{lem.twins} For three distinct vertices $a$, $b$, $c$ in a non-trivial geometric dominant graph, if $\ov{ab} = \ov{ac}$, and $a \sim b$, $a \sim c$, then $b$ and $c$ are twins.
\end{lem}

\begin{proof} First of all $b \not\sim c$, otherwise $ab=bc=ca=1$ and $b \not\in \ov{ac}$. Now suppose there is another vertex $d$ such that $b \sim d$, we are going to show that $c \sim d$.

{\em Case 1:} $a \sim d$, so $d \not\in \ov{ab} = \ov{ac}$; with $ad = ac = 1$, we must have $cd = 1$.

{\em Case 2:} $a \not\sim d$, so $d \in \ov{ab} = \ov{ac}$. By Lemma \ref{lem.bridge}, there is $b'$ such that $b'$ is adjacent to $a$, $b$, and $d$. So $b' \not\in \ov{ab} = \ov{ac}$ therefore $b' \neq c$ and $b'c = 1$ (otherwise $b'a = ac = 1$ implies $b' \in \ov{ac}$).
Now $cd \leq cb' + b'd = 2$, $ad = 2$, and $ac = 1$. Since $d \in \ov{ab} = \ov{ac}$, we must have $cd = 1$.
\end{proof}

Before we go to the final theorem, we have the last lemma that holds for any connected graphs.

\begin{lem}\label{lem.twin_non_edge} In any connected graph $G$, if $a$ and $b$ is a pair of non-adjacent vertices, and if $a$ has a twin $a'$
and $b$ has a twin $b'$ such that $a' \neq b$ and $b' \neq a$ ($a'$ and $b'$ might be the same), then
$\ov{ab}$ is a unique line in the sense that $\ov{cd} = \ov{ab}$ if and only if $\{a, b\} = \{c, d\}$.
\end{lem}

\begin{proof} Suppose $\ov{ab} = \ov{cd}$. First we note that, since $G$ is connected, $aa' = 2$. And
\begin{equation}\label{eq.twin_dists}
\forall z \not\in \{a, a'\}, za = za'.
\end{equation}
In particular
(since $a \not\sim b$) $ab = a'b > 1$, together with $aa'=2$, we have $a' \not\in \ov{ab}$.
So $a' \not\in \{c, d\}$.
If we also have $a \not\in \{c, d\}$, then (\ref{eq.twin_dists}) implies that $ac = a'c$ and $ad = a'd$,
and so $\ov{cd}$ contains both $a$ and $a'$ or none. Contradiction with the fact that $a' \not\in \ov{ab} = \ov{cd}$.

So, $a \in \{c, d\}$. By the similar argument, $b \in \{c, d\}$. So $\{a, b\} = \{c, d\}$.
\end{proof}

\begin{thm}\label{thm.main} $g(n) \in \Omega(n^{4/3})$.
\end{thm}

\begin{proof} Suppose $G = (V, E)$ and $|V| = n$. The pairs of twins is an equivalence relation over $V$ that gives a partition of the vertices. Let $X_1$ be a set where we pick one vertex from each twin class; and let $X_2 = V - X_1$ so any vertex in $X_2$ has a twin in $X_1$.

{\em Case 1:} $|X_1| \geq n / 2$. Let
\[X_0 = \{ a \in X_1 : \deg(a) \geq \frac{n-1}{2}. \}\]
If $|X_0| \leq n / 4$, then the number of edges in the complement of $G$ is at least $n(n-1)/8$, and Lemma \ref{lem.many_missing_edges}
guarantees $\Omega(n^{4/3})$ lines. Otherwise $|X_0| >  n / 4$, partition the $\binom{|X_0|}{2}$ pairs of vertices in $X_0$ according to the lines they generate. If each part has size less than $n^{2/3}$ then we are done.
Otherwise, there is a line $L$ generated by more than
$n^{2/3}$ pairs in $X_0$. If the generator graph $H(L)$ is a star, Lemma \ref{lem.star_generator} guarantees $\Omega(n^{4/3})$ lines. Otherwise, by Lemma \ref{lem.bip_blocks}, the generating pairs have common distance $d(L)$ in $G$. If $d(L) > 1$, it is a matching and Lemma \ref{lem.parallel_non_edge} gives us $\Omega(n^{4/3})$ lines. If $d(L) = 1$, note that there are no twins in $X_1$ and by Lemma \ref{lem.twins}, $H(L)|_{X_0}$ must be a matching.
So we have $2t$ ($t > n^{2/3}$) distinct vertices $a_i$, $b_i$, $1 \leq i \leq t$ such that $\ov{a_ib_i} = L$ and $a_i \sim_G b_i$.
We are going to find a distinct line $L_{ij}$ for each $i < j$ such that
\[ \{a_k, b_k\} \subseteq L_{ij} \text{ iff } k \in \{i, j\} . \]
Fix any pair $i < j$, by Lemma \ref{lem.parallel} $a_ib_i = a_jb_j = 1$, $a_ia_j = b_ib_j = x$, and $a_ib_j = a_jb_i = y$.
It is clear from the definition of $X_0$ that any pair in $X_0$ has distance at most 2. In order to have $\ov{a_ib_i} = \ov{a_jb_j}$,
$a_j \in \ov{a_ib_i}$ so,
one of $x$ and $y$ is 1, and the other is 2. We define $L_{ij} = \ov{a_ia_j}$ if $a_ia_j = 2$, otherwise $L_{ij} = \ov{a_ib_j}$.
Clearly $\{a_k, b_k\} \subseteq L_{ij}$ when $k \in \{i, j\}$.
For $k \not\in  \{i, j\}$, same reason shows that $a_ia_k = 2$ or $a_ib_k = 2$. In the former case $a_k \not\in L_{ij}$, in the latter
$b_k \not\in L_{ij}$.

{\em Case 2:} $|X_1| < n / 2$. Let $S$ be the largest clique in $X_2$. And write
\[
  E_1 = \{ \{ a, b \} : a, b \in S, a \neq b \},
  E_2 = \{ \{ a, b \} : a, b \in X_2, a\neq b, a \not\sim b \}.
\]
Tur\'an's theorem (\cite{Turan}) guarantees that $|S| \geq |X_2|^2 / (2|E_2| + |X_2|)$, and note that $|E_1| = \binom{|S|}{2}$, so at least one of $|E_1|$ and $|E_2|$ is of order $\Omega(|X_2|^{4/3}) = \Omega(n^{4/3})$. We conclude the proof by pointing out that any pair in $E_1 \cup E_2$ generates a distinct line. Indeed, every pair in $E_2$ generates a distinct line in $G$ by Lemma \ref{lem.twin_non_edge}, and every pair in $E_1$ generates a line that distinctly intersects $S$ only on that pair.

\end{proof}

Part of the proof of Theorem \ref{thm.main} resembles the proof of the lower bound on the number of lines in a metric space with distances $0, 1, 2$ in \cite{Chini_Chvatal}. In fact we do not know if all the non-trivial geometric dominant graphs are of diameter 2.

\section{Discussions}

In the beginning of this work, we proved some properties of the geometric dominant graphs
and found that small non-trivial geometric dominant graphs are rare. This led us to the illusion
that the truth might be similar to those in the classic theorems of Erd\H{o}s-R\'enyi-S\'os (\cite{ERS}) and of
Hoffman-Singleton (\cite{HS}). In particular, one may guess that any such graph must have a center,
and the number of such graphs is small or even zero for big enough $n$.
Most of the illusions were refuted by the delightful surprise of Theorem \ref{thm.super}.
Yet, we do not know

\begin{ques} True or false? Every non-trivial geometric dominant graph has diameter 2?
\end{ques}

The super geometric dominant graphs is an interesting subject by its own right.
Fact \ref{fact.explode} tells that a super geometric dominant graph with more edges
explodes to a geometric dominant graph with less lines. We have

\begin{ques} What is the maximum / minimum number of edges a super geometric dominant graph can have?
\end{ques}

The construction in Theorem \ref{thm.super_super} shows the existence of super geometric dominant graphs
that missed only $O(n \ln n)$ edges. Is that the best possible?
In fact, besides the random graphs, we do not know any constructive description of a (family of)
super geometric dominant graphs. In our calculation, the random graph must have hundreds of vertices
to become super geometric dominant. It is also interesting to study whether such graphs of small sizes exist.

In this work we only focused on the geometric dominant graphs. We would
like to study the geometric dominant metric spaces in the future.
In particular, here is the special case of Chen-Chv\'atal conjecture:

\begin{ques} True or false? Every geometric dominant metric space $(V, \rho)$ where $V$
  is not a line has at least $\Omega(|V|)$ lines?
\end{ques}

\section*{Acknowledgement}

We would like to thank Va\v sek Chv\'atal for the discussions and comments on this work,
and for his inspring role in the study of lines in hypergraphs and metric spaces.
X. Chen would like to thank Pierre Aboulker, Rohan Kapadia, V\'aclav Kratochvíl, Ben Seamone, and Cathryn Supko
for the nice discussions on this and related problems during his short visit to Concordia University.

\appendix

\section{Almost all graphs are super geometric dominant}

Here we give the proof of Theorem \ref{thm.super} that the random graph $\mathcal{G}_{n, p}$ is super geometric dominant almost surely
for a big range of $p$.
As we suggested, the proof is quite similar to that of
Theorem \ref{thm.super_super}. It is only slightly more complicated because we are dealing with
a big range of $p$ instead of just $p = 1/2$.

\begin{thm*} When $n \rightarrow \infty$ and $p$ is a function such that
  \[p \in \omega\left(\sqrt[3]{\frac{\ln n}{n}} \right) \;\; \text{ and } \;\; 1 - p \in \omega\left(\sqrt{\frac{\ln n}{n}} \right), \]
the random graph $\mathcal{G}_{n, p}$ is super geometric dominant (therefore geometric dominant) almost surely.
\end{thm*}

\begin{proof}
Let $q = 1 - p$. It is clear that $p^2 + q^2 \geq 1/2$, and, because one of $p$ and $q$ is at least $1/2$, when $p$ grows with $n$ under the condition specified by our statement, $p^iq^j \geq \min(p^i, q^j) / 8 \in \omega(\ln n / n) $ for any integers $1 \leq i \leq 3$ and $1 \leq j \leq 2$.

We call a three-vertex set {\em tight} if they induce two edges. Note that in the metric space induced by a graph with diameter 2,
three vertices are collinear if and only if they form a tight set. Given any three vertices $a$, $b$, and $z$, and let $T$ be the event that $\{a, b, z\}$ is tight, we have
\begin{equation}
\begin{split}
\prob(T | a \not\sim b) = p^2, \;\;\; \prob(\ov{T} |  a \not\sim b) = 1 - p^2 = (1+p)q \\
\prob(T | a \sim b) = 2pq, \;\;\; \prob(\ov{T} |  a \sim b) = 1 - 2pq = p^2 + q^2
\end{split}
\end{equation}

We define and bound (the probability of the complement of) the following events.

For any two vertices $a$ and $b$, $D_{ab}$ is the event that there is another vertex $z$ such that $a \sim z$ and $b \sim z$. And let $D$ be the intersection of all the $D_{ab}$'s, so $D$ implies that the graph has diameter at most 2.
\[\prob(\ov{D_{ab}}) = (1 - p^2)^{n-2} \leq \exp(-p^2(n-2)).\]

For any two vertices $a$ and $b$, $N_{ab}$ is the event that there is another vertex $z$ such that $a \sim z$ and $b \not\sim z$.
\[\prob(\ov{N_{ab}}) = (1 - pq)^{n-2} \leq \exp(-pq(n-2)).\]

For any three vertices $a$, $b$, and $c$, $L_{abc}$ is the event that either none of $b$ and $c$ is adjacent to $a$, or there is another vertex $z$ such that $\{a, b, z\}$ is tight and $\{a, c, z\}$ is not tight. Note that any outcome in $D \cap L_{abc}$ has the property that lines $\ov{ab} \not\subseteq \ov{ac}$.

\begin{figure}[h]
  \begin{center}
    \includegraphics[width=5.2in]{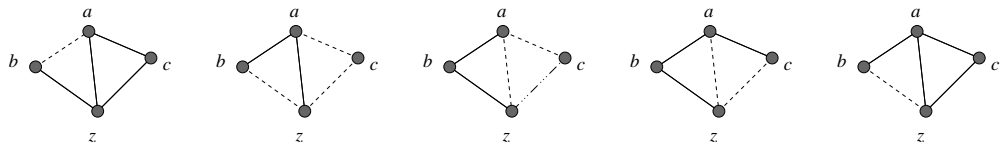}
  \end{center}

  \caption{The cases for $L_{abc}$. Straight lines denote adjacent vertices, dashed lines denote non-adjacent vertices, mixed lines means two vertices can be adjacent or non-adjacent.}\label{fig.super}
\end{figure}

As in Figure \ref{fig.super}, we have
\[\prob(\ov{L_{abc}} | a \not\sim b, a \sim c) = (1 - p^3)^{n-3} \leq \exp(-p^3(n-3)).\]
\[\prob(\ov{L_{abc}} | a \sim b, a \not\sim c) = (1 - pq^2 - pq)^{n-3} \leq \exp(-pq(n-3)).\]
\[\prob(\ov{L_{abc}} | a \sim b, a \sim c) = (1 - pq^2 - p^2q)^{n-3} = (1 - pq)^{n-3} \leq \exp(-pq(n-3)).\]

For any four distinct vertices $a$, $b$, $c$, and $d$, $L_{abcd}$ is the event that there is a vertex $z \not\in \{a, b, c, d\}$ such that $\{a, b, z\}$ is tight and $\{c, d, z\}$ is not tight.
\[\prob(\ov{L_{abcd}} | a \not\sim b, c \not\sim d) = (1 - p^2(1+p)q)^{n-4} \leq \exp(-p^2q(n-4)).\]
\[\prob(\ov{L_{abcd}} | a \not\sim b, c \sim d) = (1 - p^2(p^2+q^2))^{n-4} \leq \exp(-p^2q^2(n-4)).\]
\[\prob(\ov{L_{abcd}} | a \sim b, c \not\sim d) = (1 - 2pq(1+p)q)^{n-4} \leq \exp(-2pq^2(n-4)).\]
\[\prob(\ov{L_{abcd}} | a \sim b, c \sim d) = (1 - 2pq(p^2+q^2))^{n-4} \leq \exp(-pq(n-4)).\]

For any two vertices $a$ and $b$, $E_{ab}$ is the event that there is another vertex $z$ such that $z \sim a$ and $\{a, b, z\}$ is not tight.
\[\prob(\ov{E_{ab}} | a \not\sim b) = (1 - pq)^{n-2} \leq \exp(-pq(n-2)).\]
\[\prob(\ov{E_{ab}} | a \sim b) = (1 - p^2)^{n-2} \leq \exp(-p^2(n-2)).\]

For any two vertices $a$ and $b$, $E'_{ab}$ is the event that either $a \not\sim b$, or there is another vertex $z$ such that $z \not\sim a$ and $z \sim b$. Note that any outcome in $D \cap E'_{ab}$ has the property that $\ov{ab} \not\subseteq N^*(a)$.
\[\prob(\ov{E'_{ab}}) = p(1 - pq)^{n-2} \leq \exp(-pq(n-2)).\]

For any three vertices $a$, $b$, and $c$, $E_{abc}$ is the event that there is another vertex $z$ such that $z \sim a$ and $\{b, c, z\}$ is not tight.
\[\prob(\ov{E_{abc}} | b \not\sim c) = (1 - p(1+p)q)^{n-3} \leq \exp(-pq(n-3)).\]
\[\prob(\ov{E_{abc}} | b \sim c) = (1 - p(p^2 + q^2))^{n-3} \leq \exp(-\frac{p}{2}(n-3)).\]

For any three vertices $a$, $b$, and $c$, $E'_{abc}$ is the event that there is another vertex $z$ such that $z \not\sim a$ and $\{b, c, z\}$ is tight.
\[\prob(\ov{E'_{abc}} | b \not\sim c) = (1 - qp^2)^{n-3} \leq \exp(-p^2q(n-3)).\]
\[\prob(\ov{E'_{abc}} | b \sim c) = (1 - q \cdot 2pq)^{n-3} \leq \exp(-2pq^2(n-3)).\]

Thus we defined $O(n^4)$ events, and the probability of the complement of each is bounded above by $\exp(-Cp^3n)$ or $\exp(-Cq^2n)$ where $C$ is some constant. We conclude our proof by pointing out that any outcome in the intersection of all the events is a super geometric dominant graph.

\end{proof}

\section{Some easy proofs of weaker lower bounds on lines in geometric dominant graphs}

Lemma \ref{lem.star_generator} already provides a lower bound of $\Theta(n^{2/3})$ lines in any non-trivial geometric dominant graph --- consider any vertex $v$
and all the lines $\ov{vw}$, if none of the lines is generated $n^{1/3}$ times, then there are at least $(n-1) / n^{1/3}$ lines; otherwise, by Lemma \ref{lem.star_generator}, there
are $\Theta(n^{2/3})$ lines.

From Lemma \ref{lem.bip_blocks} we had a short proof for the linear lower bound.

\begin{thm} $g(n) \in \Omega(n)$.
\end{thm}

\begin{proof}
Let $G = (V, E)$ be a non-trivial geometric dominant graph and $|V| = n$.

{\em Case 1.} For every line $L$, the number of edges in $H(L)$ is at most $n$.
Note that all the generator graphs form an edge partition of the complete graph $K_n$.
So there are $\Omega(n)$ lines.

{\em Case 2.}
There is a line $L$ such that $H(L)$ has more than $n$ edges. By Lemma \ref{lem.bip_blocks}, the components
of $H(L)$ are complete bipartite graphs $(A_i, B_i)$ $i = 1, 2, ..., t$. We may arrange the blocks so that (1) $|A_i| \geq |B_i|$, (2) $|A_i| \geq |A_j|$ whenever $i < j$.
So there is a $1 \leq t^* \leq t+1$ such that $|A_i| > 2$ if and only if $i < t^*$. The number of edges in $H(G)$ is
\[ \sum_{i = 1}^t |A_i||B_i| \leq \sum_{i < t^*} |A_i||B_i| + \sum_{i \geq t^*} |A_i||B_i| \leq \sum_{i < t^*} 4 \binom{|A_i|}{2} + 4t. \]
Because there are at most $n / 2$ non-isolated connected components, $4t \in O(n)$. If $\sum_{i < t^*} \binom{|A_i|}{2} \in O(n)$ as well, we are done. Otherwise,
consider the any line $\ov{ab}$ for $a, b \in A_i$ for some $i < t^*$. By Lemma \ref{lem.star_generator}, $\ov{ab} \cap |A_i| = \{a, b\}$. And by Lemma \ref{lem.bip_blocks}(b),
for any other block $A_j$ ($j < t^*$), either $A_j \subseteq \ov{ab}$, or $X \cap \ov{ab} = \emptyset$. It is clear to see these are $\sum_{i < t^*} \binom{|A_i|}{2}$ distinct lines in $G$.
\end{proof}

\end{document}